\documentclass[12pt]{amsart}
\usepackage{amssymb}
\usepackage{amsfonts}
\usepackage{amsmath,latexsym,amssymb,amsfonts,amsbsy, amsthm}

\newcommand{\beq}{\begin{equation}}
\newcommand{\eeq}{\end{equation}}
\newcommand{\ben}{\begin{eqnarray}}
\newcommand{\een}{\end{eqnarray}}
\newcommand{\beno}{\begin{eqnarray*}}
\newcommand{\eeno}{\end{eqnarray*}}
\newcommand{\R}{\mathbb{R}}
\newtheorem{thm}{Theorem}[section]
\newtheorem{defi}[thm]{Definition}
\newtheorem{lem}[thm]{Lemma}
\newtheorem{prop}[thm]{Proposition}
\newtheorem{coro}[thm]{Corollary}

\begin{document}

\title[Axially symmetric solution]{Axially symmetric solutions of Allen-Cahn equation with finite Morse index$^*$}
\author[C. Gui, K. Wang and J. Wei]{
Changfeng Gui$^\dag$,
Kelei Wang$^\ddag$ and Juncheng Wei$^\S$}
\thanks{$^\dag$ Department of Mathematics
University of Texas at San Antonio
San Antonio, TX 78249 USA.
{Email: changfeng.gui@utsa.edu}}

\thanks{
$^\ddag$School of Mathematics and Statistics \& Computational
Science Hubei Key Laboratory, Wuhan University, Wuhan 430072, China.
{Email: wangkelei@whu.edu.cn}. }

\thanks{$^\S$Department of Mathematics, University of British
Columbia, Vancouver, B.C., Canada, V6T 1Z2.
{Email: jcwei@math.ubc.ca}.}

\thanks{$\ddag$  The research of K. Wang was supported
by NSFC. 11871381 and 11631011. The research of J. Wei is partially supported by NSERC of Canada. }
\date{\today}

\begin{abstract}
In this paper we study axially symmetric solutions of Allen-Cahn equation with finite Morse index. It is shown that there does not exist such a solution in dimensions between $4$ and $10$. In dimension $3$, we prove that these solutions have finitely many ends. Furthermore, the solution has exactly two ends if its Morse index equals $1$.
\end{abstract}

\subjclass{  }

\maketitle
\renewcommand{\theequation}{\thesection.\arabic{equation}}
\setcounter{equation}{0}

\section{Introduction}
\setcounter{equation}{0}

\renewcommand{\theequation}{\thesection.\arabic{equation}}
\setcounter{equation}{0}

In this paper we study axially symmetric solutions of the Allen-Cahn equation
\begin{equation}\label{Allen-Cahn}
\Delta u=W^\prime(u), \quad \mbox{ in } \R^{n+1}.
\end{equation}
Here $W(u)$ is a general double well potential, that is, $W\in C^3([-1,1])$ satisfying
\begin{itemize}
\item $W>0$ in $(-1,1)$ and $W(\pm1)=0$;
\item  $W^\prime(\pm1)=0$ and $W^{\prime\prime}(-1)=W^{\prime\prime}(1)=2$;
\item $W$ is even and $0$ is the unique critical point of $W$ in $(-1,1)$.
\end{itemize}
A typical model is given by $W(u)=(1-u^2)^2/4$.

For this class of double well potential $W$, there exists a unique solution  to the following one dimensional problem
\begin{equation}\label{1d problem}
g^{\prime\prime}(t)=W^\prime(g(t)),  \quad \ g(0)=0 \ \ \quad \mbox{and }\ \lim_{t\to\pm\infty}g(t)=\pm 1.
\end{equation}
Moreover, as $t\to\pm\infty$, $g(t)$ converges exponentially to $\pm1$ and  the following quantity is well defined
\[\sigma_0:=\int_{-\infty}^{+\infty}\left[\frac{1}{2}g^\prime(t)^2+W(g(t))\right]dt\in (0,+\infty).\]

In fact,  as $t\to\pm\infty$, the following expansions hold: there exists a positive constant $A$ such that
for all $|t|$ large,
\begin{equation*}
  \begin{cases}
   &g(t)=(1- A e^{-\sqrt{2} |t|} ) sign(t) +O(e^{-2 \sqrt{2} |t|}),  \\
    & g^\prime(t)= \sqrt{2}A  e^{- \sqrt{2} |t|}+O(e^{-2\sqrt{2}  |t|}), \\
 & g^{\prime\prime}(t)=-2 A  e^{-  \sqrt{2}|t|}+O(e^{-2\sqrt{2}  |t|}).
  \end{cases}
\end{equation*}

Denote points in $\R^{n+1}$ by $(x_1,\cdots,x_{n},z)$ and let $r:=\sqrt{x_1^2+\cdots+x_{n}^2}$.

\begin{defi}
\begin{itemize}
\item A function $u$ is axially symmetric if $u(x_1,\cdots,x_{n},z)=u(r,z)$.

\item  A solution of \eqref{Allen-Cahn} is stable in a domain $\Omega\subset\R^{n+1}$ if for any $\varphi\in C_0^\infty(\Omega)$,
\[\mathcal{Q}_\Omega(\varphi):=\int_{\Omega}\left[|\nabla\varphi|^2+W^{\prime\prime}(u)\varphi^2\right]\geq 0.\]
\item A solution of \eqref{Allen-Cahn} has finite Morse index in $\R^{n+1}$ if
\[\sup_{R>0}\mbox{dim}\left\{\mathcal{X}\subset C_0^\infty(B_R(0)): \mathcal{Q}\lfloor_{\mathcal{X}}<0\right\}<+\infty.\]
It is well known that  the finite Morse index condition is equivalent to the condition of being stable outside a compact set (see, e.g., \cite{6}).
\end{itemize}
\end{defi}

\begin{defi}
An axially symmetric solution of \eqref{Allen-Cahn} has finitely many ends if for some $R>0$,
\begin{itemize}
\item $u\neq 0$ in $B_R^{n}(0)\times\{|z|>R\}$;
\item outside $\mathcal{C}_R:=B_R^{n}(0)\times\R$, $\{u=0\}$ consists of finitely many graphs $\Gamma_\alpha$, where
\[\Gamma_\alpha=\left\{z=f_\alpha(r)\right\}, \quad \alpha=1,\cdots, Q,\]
and $f_1<\cdots<f_Q$.
\end{itemize}
\end{defi}

Our first main result is
\begin{thm}\label{main result 1.2}
If $3\leq n \leq 9$, any axially symmetric solution of \eqref{Allen-Cahn}, which  is stable outside a cylinder $\mathcal{C}_R$,   depends only on $z$.
\end{thm}
In other words,  the solution has exactly one end and it is one dimensional, i. e. all of its level sets are hyperplanes of the form
 $\{z=t\}$. Therefore  for $3\leq n \leq 9$, there does not exist   axially symmetric solutions which is stable outside a cylinder, except the trivial ones (i.e., constant solutions $\pm 1$  and $g$ in \eqref{1d problem}).

The dimension bound in this theorem is {\it sharp}. On one hand, if $n\geq 10$, there do exist \emph{stable}, axially symmetric solutions of \eqref{Allen-Cahn} in $\R^{n+1}$ with two ends, see Agudelo-Del Pino-Wei \cite{1}. (The two-end solutions constructed in the paper for $3 \le n\le 9 $ are also shown to be unstable by a different argument. Our proof of Theorem \ref{main result 1.2} will rely on an idea of Dancer and Farina \cite{4}.)
On the other hand, nontrivial axially symmetric solutions with finite Morse index in $\R^3$ also exist. (See del Pino-Kowalczyk-Wei \cite{5}.) However we show that
\begin{thm}\label{main result 1.1}
If $n=2$, an axially symmetric solution of \eqref{Allen-Cahn} with finite Morse index  has finitely many ends.
Moreover, there exists a constant $C$ such that for any $x\in\R^3$ and $R>0$,
\begin{equation}\label{energy growth}
\int_{B_R(x)}\left[\frac{1}{2}|\nabla u|^2+W(u)\right]\leq CR^2.
\end{equation}
\end{thm}

Concerning solutions with a low Morse index we first show that
\begin{thm}\label{main result 2}
If $n=2$, any axially symmetric, stable solution of \eqref{Allen-Cahn} depends only on $z$.
\end{thm}
Next we prove that
\begin{thm}\label{main result 3}
Any axially symmetric  solution of \eqref{Allen-Cahn} with Morse index $1$ in $\R^3$ has exactly two ends.
\end{thm}

Two end solutions in $\R^3$ have been studied in detail in  Gui-Liu-Wei \cite{9}. They showed  that for each $k \in (\sqrt{2}, +\infty)$ there exists two-ended axially symmetric solutions whose  zero level set approximately look like  $ \{ z= k \log r \}$. Parallel to R. Schoen's result in minimal surfaces \cite{11}, one may ask the following natural question:

\medskip

\noindent
{\bf Conjecture:} All two-ended solutions to Allen-Cahn equation in $\R^3$ must be axially symmetric.

\medskip

\medskip

We introduce some notations used in the proof of  Theorems \ref{main result 1.2}-\ref{main result 3}.
Taking $(r,z)$ as coordinates in the plane, after an even extension to $\{r<0\}$, an axially symmetric function $u$ can be viewed as a smooth function defined on $\R^2$. Now \eqref{Allen-Cahn} is written as
\begin{equation}\label{equation}
u_{rr}+\frac{n-1}{r}u_r+u_{zz}=W^\prime(u).
\end{equation}
We use subscripts to denote differentiation, e.g. $u_z:=\frac{\partial u}{\partial z}$. A nodal domain of $u_z$ is a connected component of $\{u_z\neq0\}$. Sometimes we will identify various objects in $\R^{n+1}$ with the corresponding ones in the $(r,z)$-plane, if they have axial symmetry.

\medskip

To prove  Theorems \ref{main result 1.2}-\ref{main result 3} we  follow from a strategy used by the second and the third authors \cite{17}. One of the main difficulties is the possibility of  an infinite tree of nodal domains of $ \frac{\partial u}{\partial z} (r, z)$. Here we explore the decaying properties of the curvature to exclude this scenario.

\medskip

The remaining part of this paper is organized as follows. In Section \ref{sec curvature decay} we give a curvature decay estimate on level sets of $u$. This curvature estimate allows us to determine the topology and geometry of ends in Section \ref{sec geometry of ends}. In Section \ref{sec refined asymptotics} we show that interaction between different ends is modeled by a Toda system. The case $3\leq n \leq 9$ is analysed in Section \ref{sec high dim case}, while Section \ref{sec 3 dim case} is devoted to the proof of the $n=2$ case. Finally, Theorem \ref{main result 2} and Theorem \ref{main result 3} are proved in Section \ref{sec low Morse index case}.

\section{Curvature decay}\label{sec curvature decay}
\setcounter{equation}{0}
In this section we establish a technical result on curvature decay of level sets of $u$.

Let us first recall several results on stable solutions of \eqref{Allen-Cahn}. By \cite{12}, given a  domain $\Omega\subset\R^{n+1}$, the  condition that $\mathcal{Q}(\varphi)\geq0$ for all $\varphi\in C_0^\infty(\Omega)$ is equivalent to the following Sternberg-Zumbrun inequality
\begin{equation}\label{S-Z inequality}
  \int_\Omega |\nabla\varphi|^2|\nabla u|^2\geq\int_\Omega \varphi^2|B(u)|^2|\nabla u|^2, \quad \forall \varphi\in C_0^\infty(\Omega).
\end{equation}
Here
\begin{equation}\label{Sternberg-Z}
|B(u)|^2:=\frac{|\nabla^2u|^2-|\nabla|\nabla u||^2}{|\nabla
u|^2}=|A|^2+|\nabla_T\log|\nabla u||^2,
\end{equation}
where $A$ is the second
fundamental form of the level set of $u$ and $\nabla_T$ is the tangential derivative
along the level set.

The following \emph{Stable De Giorgi} theorem in dimension $2$ is well known, see \cite{8}.
\begin{thm}\label{thm stable De Giorgi}
Suppose $u$ is a stable solution of \eqref{Allen-Cahn} in $\R^2$. Then $u$ is one dimensional. In particular, $|B(u)|^2\equiv 0$.
\end{thm}

Using this theorem we show
\begin{prop}\label{prop limit at infinity}
Suppose $u$ is an axially symmetric solution of \eqref{equation} in $\R^{n+1}$, which is stable outside a cylinder $\mathcal{C}_R$. Then for any $R_i\to+\infty$ and $z_i\in\R$, after passing to a subsequence, $u_i(r,z):=u(R_i+r,z_i+z)$ converges to a one dimensional solution of \eqref{Allen-Cahn} in $C^2_{loc}(\R^2)$.
\end{prop}
\begin{proof}
By standard elliptic estimates we can assume $u_i$ converge to $u_\infty$ in $C^2_{loc}(\R^2)$.  Passing to the limit in \eqref{equation} we see $u_\infty$ is a solution of \eqref{Allen-Cahn} in $\R^2$.

 Because $u$ is axially symmetric and stable outside $\mathcal{C}_R$, there exists an axially symmetric function $\varphi$ which is positive outside $\mathcal{C}_R$ such that
\[\varphi_{rr}+\frac{n-1}{r}\varphi_r+\varphi_{zz}=W^{\prime\prime}(u)\varphi, \quad \mbox{outside } \ \mathcal{C}_R.\]
Define
\[\varphi^i(r,z):=\frac{1}{\varphi(R_i,z_i)}\varphi(R_i+r,z_i+z).\]
For any $R>0$, it satisfies
\[\varphi^i_{rr}+\frac{n-1}{R_i+r}\varphi^i_r+\varphi^i_{zz}=W^{\prime\prime}(u_i)\varphi^i, \quad \mbox{in } \ B_R^2(0).\]
By definition, $\varphi^i(0)=1$ and $\varphi^i>0$. Then by Harnack inequality and standard elliptic estimates, after passing to a subsequence we can take a limit $\varphi^i\to\varphi^\infty$ in $C^2_{loc}(\R^2)$. Here $\varphi^\infty$ satisfies
\[\varphi^\infty_{rr}+\varphi^\infty_{zz}=W^{\prime\prime}(u_\infty)\varphi^\infty, \quad \varphi^\infty>0 \quad \mbox{in } \ \R^2.\]
Hence $u_\infty$ is a stable solution of \eqref{Allen-Cahn} in $\R^2$. By Theorem \ref{thm stable De Giorgi}, $u_\infty$ is one dimensional.
\end{proof}
\begin{coro}\label{coro 2.1}
Suppose $u$ is an axially symmetric solution of \eqref{equation}  in $\R^{n+1}$, which is stable outside a cylinder $\mathcal{C}_R$. For any $b\in(0,1)$, there exists an $R(b)>0$ such that $|\nabla u|\neq 0$ in $\{|u|<1-b\}\setminus\mathcal{C}_{R(b)}$. Moreover, if $x\in\{|u|<1-b\}\setminus\mathcal{C}_{R(b)}$ and $x\to\infty$,
\[|B(u)(x)|\to0.\]
\end{coro}

The main technical tool we need in this paper is the following decay estimate on $|B(u)|^2$.
\begin{thm}\label{thm curvature decay}
Suppose $u$ is an axially symmetric solution of \eqref{equation}  in $\R^{n+1}$, which is stable outside a cylinder $\mathcal{C}_R$. For any $b\in(0,1)$, there exists a constant $C(b)$ such that in $\{|u|<1-b\}\setminus \mathcal{C}_{R(b)}$,
\[|B(u)(r,z)|^2\leq C(b)r^{-2}\]
and
\[|H(u)(r,z)|\leq C(b)r^{-2}\left(\log\log r\right)^2.\]
\end{thm}
In the above $H(u)(r,z)$ denotes the mean curvature of the level set $\{u=u(r,z)\}$ at the point $(r,z)$.
The proof of this theorem is similar to the  two dimensional case in \cite{17}. By a blow up method, it is  reduced to the second order estimate established in \cite{18}. Note that here we do not impose any condition on $n$, because as in the proof of Proposition \ref{prop limit at infinity}, the limiting problem after blow up is essentially a two dimensional problem and then the estimate in \cite{18} is applicable.

\section{Geometry of ends}\label{sec geometry of ends}
\setcounter{equation}{0}

In this section $u$ denotes an axially symmetric solution of \eqref{equation}  in $\R^{n+1}$, $n\geq 2$, which is stable outside a cylinder $\mathcal{C}_R$.
Here and henceforth, a constant $b\in(0,1)$ will be fixed and notations in the previous section will be kept. Take a constant $R_1 >R(b)$ so that it satisfies
 \begin{equation}\label{3.1}
C(b)R_1 ^{-2}\left(\log\log R_1 \right)^2<R_1^{-1}.
\end{equation}

 By Theorem \ref{thm curvature decay}, $\{u=0\}\setminus \mathcal{C}_{R_1 }=\cup_\alpha \Gamma_\alpha$, where $\alpha\in\mathcal{A}$ is the index. For each $\alpha$, $\Gamma_\alpha$ is a connected smooth embedded hypersurface with or without boundary.
Furthermore, $\Gamma_\alpha\cap\Gamma_\beta=\emptyset$ if $\alpha\neq\beta$. Finally, since $u$ is axially symmetric, for each $\alpha\in\mathcal{A}$, $\Gamma_\alpha$ is also axially symmetric. As a consequence, $\Gamma_\alpha$ can be viewed as a smooth curve in the $(r,z)$ plane.

Viewing  $\Gamma_\alpha$ as a smooth curve in the $(r,z)$ plane and $r$ as a function defined on $\Gamma_\alpha$, we have
\begin{lem}\label{lem critical points of r}
Every critical point of $r$ in the interior of $\Gamma_\alpha$ is a strict local minima.
\end{lem}
\begin{proof}
Assume by the contrary, there exists a point $(r_\ast,z_\ast)$ in the interior of one $\Gamma_\alpha$, which is a critical point of $r$ but not a strict local minima. By Corollary \ref{coro 2.1}, in a neighborhood of  $(r_\ast,z_\ast)$, $\Gamma_\alpha=\{r=f_\alpha(z)\}$. By our assumptions, $f_\alpha(z_\ast)=r_\ast$, $f_\alpha^\prime(z_\ast)=0$ and $f_\alpha^{\prime\prime}(z_\ast)\leq 0$. Hence
\[ H_{\Gamma_\alpha}(r_\ast,z_\ast)\geq\frac{1}{r_\ast}.\]
In view of \eqref{3.1}, this is a contradiction with Theorem \ref{thm curvature decay}.
\end{proof}

Since $\Gamma_\alpha$ is a connected smooth curve with end points (if there are) in $\partial\mathcal{C}_{R_1}$, by this lemma we see there is no local maxima and at most one local minima  of $r$ in the interior of $\Gamma_\alpha$. There are two cases:
\begin{itemize}
\item[{\bf  Type I.}] $\Gamma_\alpha$ is diffeomorphic to $[0,+\infty)$ and it has exactly one end point on $\partial\mathcal{C}_{R_1}$;
\item[{\bf Type II.}] $\Gamma_\alpha$ is diffeomorphic to $(-\infty,+\infty)$ and its boundary is empty.
\end{itemize}

If $\Gamma_\alpha$ is of type I, $r$ is a strictly increasing function with respect to a parametrization of $\Gamma_\alpha$. Hence it can be represented by the graph $\{z=f_\alpha(r)\}$, where $f_\alpha\in C^4[R_1,+\infty)$. (Higher order regularity on $f_\alpha$ follows by applying the implicit function theorem to $u$.)

If $\Gamma_\alpha$ is of type II, there exists a point $(R_\alpha,z_\alpha)$, which is the unique  minima of $r$ on $\Gamma_\alpha$.
As in Type I case, $\Gamma_\alpha\setminus\{(R_\alpha,z_\alpha)\}=\Gamma_\alpha^+\cup\Gamma_\alpha^-$, where $\Gamma_\alpha^\pm$ can be represented by two graphs $\{z=f_\alpha^\pm(r)\}$. Here $f_\alpha^+>f_\alpha^-$ on $(R_\alpha,+\infty)$ and $f_\alpha^+(R_\alpha)=f_\alpha^-(R_\alpha)=z_\alpha$.

\begin{prop}\label{prop nonexistence of type II}
There exists a constant $R_2>R_1$ such that for any type II end $\Gamma_\alpha$, it holds that $R_\alpha< R_2$.
\end{prop}
\begin{proof}
Assume by the contrary, there exists a sequence of type II ends $\Gamma_k$ such that $R_k\to+\infty$.

By Theorem \ref{thm curvature decay}, the rescalings $\Sigma_k:=R_k^{-1}\left[\Gamma_k-(0,z_k)\right]$ have uniformly bounded curvatures and their mean curvatures converge to $0$ uniformly. By standard elliptic estimates, after passing to a subsequence of $k$, $\Sigma_k$ converges to an axially symmetric, smooth minimal hypersurface $\Sigma_\infty$. Moreover, there exist two functions $f^\pm_\infty\in C^2((1,+\infty))$ such that
\[\Sigma_\infty\setminus\{(1,0)\}=\left\{(r,z): z=f_\infty^\pm(r)\right\}.\]
Hence $\Sigma_\infty$ is the standard catenoid. By \cite{13}, it is unstable. (Indeed, its Morse index is exactly $1$.)

On the other hand, we claim that $\Sigma_\infty$ inherits the stability from $u$, thus arriving at a contradiction. Indeed, let
$u_k(r,z):=u(R_kr, R_k (z_k+z))$. It is a solution of the singularly perturbed Allen-Cahn equation
\[\Delta u_k=R_k^2 W^\prime(u_k).\]
Since $u$ is stable outside $\mathcal{C}_{R_1}$, $u_k$ is stable outside $\mathcal{C}_{R_1/R_k}$. Note that $\Sigma_k$ is a connected component of $\{u_k=0\}$ and it is totally located outside $\mathcal{C}_1$. Therefore we can use the method in   \cite{3} to deduce the stability of $\Sigma_\infty$. There are two cases.
\begin{itemize}
  \item Suppose there exists another connected component of $\{u_k=0\}$, denoted by $\widetilde{\Sigma}_k$, also converging to $\Sigma_\infty$ in a ball $B_r(p)$ for some $r>0$ and $p\in\Sigma_\infty$. By Theorem \ref{thm curvature decay}, $\widetilde{\Sigma}_k$ enjoys the same regularity as for $\Sigma_k$. Hence  by the axial symmetry of $\widetilde{\Sigma}_k$ and the uniqueness of catenoid, $\widetilde{\Sigma}_k$ converges to $\Sigma_\infty$ everywhere. In this case we can construct a positive Jacobi field on $\Sigma_\infty$ as in \cite[Theorem 4.1]{3}, which implies the stability of $\Sigma_\infty$
  \item Suppose there is only one such a component in a fixed neighborhood $\mathcal{N}$ of $\Sigma_\infty$. Since $\Sigma_\infty\subset\{r\geq 1\}$, we can take $\mathcal{N}\subset\{r>1/2\}$. Hence $u_k$ is stable in $\mathcal{N}$. Then for any ball $B_r(p)$ with $r>0$ and $p\in \Sigma_\infty$, there exists a constant $C>0$ such that
\[\int_{\mathcal{N}\cap B_r(p)}\left[\frac{1}{2R_k}|\nabla u_k|^2+ R_k W(u_k)\right]\leq C.\]
Because $u_k$ is stable in $\mathcal{N}\cap B_r(p)$, the stability of $\Sigma_\infty$  follows by applying the main result of \cite{14}.
\end{itemize}
The contradiction implies that $R_\alpha$ is bounded and the proposition is proven.
\end{proof}

Now $\{u=0\}\setminus\mathcal{C}_{R_2}=\cup_\alpha\Gamma_\alpha$, where each $\Gamma_\alpha$ is of Type I. Denote $\Gamma_\alpha\cap\{r=R_2\}=\{(R_2,z_\alpha)\}$. By Proposition \ref{prop limit at infinity}, after perhaps enlarging $R_2$, there is a positive lower bound for $|z_\alpha-z_\beta|$, $\forall \alpha\neq\beta$. Hence we can take the index $\alpha$ to be integers and we  will relabel indices  so that $z_\alpha<z_\beta$ for any $\alpha<\beta$. Furthermore,  we have  $f_\alpha<f_\beta$ in $[R_2,+\infty)$ for any $\alpha<\beta$.

Define the functions
\begin{equation*}
\begin{aligned}
&f_\alpha^+(r):=\frac{f_\alpha(r)+f_{\alpha+1}(r)}{2}, \quad \mbox{for } r\in[R_2,+\infty),\\
&f_\alpha^-(r):=\frac{f_\alpha(r)+f_{\alpha-1}(r)}{2}, \quad \mbox{for } r\in[R_2,+\infty).
\end{aligned}
\end{equation*}
By definition, $f_\alpha^+=f_{\alpha+1}^-$. In the above we take the convention that $f_\alpha^+(r)=+\infty$ (or $f_\alpha^-(r)=-\infty$) if there does not exist any other end lying above (respectively below) $\Gamma_\alpha$. Let
  \[\mathcal{M}_\alpha:=\left\{(r,z): f_\alpha^-(r)<z<f_\alpha^+(r), \ r>R_2\right\}.\]

The following result describes the asymptotics of $f_\alpha$ as $r\to+\infty$.
\begin{lem}\label{lem decay rate of derivatives I}
There exists a constant $C$ such that for each $\alpha$, in $[R_2,+\infty)$ we have
\begin{equation}\label{decay 1}
\begin{cases}
   & |f_\alpha(r)-f_\alpha(R_2)|\leq C  \log r\left(\log\log r\right)^2,\\
    &|f_\alpha^\prime(r)|\leq C r^{-1}\left(\log\log r\right)^2, \quad  |f_\alpha^{\prime\prime}(r)|\leq C r^{-2}\left(\log\log r\right)^2 , \\
  & |f_\alpha^{(3)}(r)|+ |f_\alpha^{(4)}(r)|\leq C r^{-2}\left(\log\log r\right)^2.
\end{cases}
\end{equation}
\end{lem}
\begin{proof}
We divide the proof  into three steps.

{\bf Step 1.}
Denote the second fundamental form of $\Gamma_\alpha$ by $A_\alpha$, the mean curvature by $H_\alpha$.
By the decay rate on $|A_\alpha|$ (see Theorem \ref{thm curvature decay}), there exists a constant $C$ such that for any $r\geq R_2$,
\begin{equation}\label{3.4}
|f_\alpha^{\prime\prime}(r)|\leq\frac{C}{r}, \quad |f_\alpha^\prime(r)|\leq C.
\end{equation}

{\bf Step 2.} For any $\lambda>0$, let $\Sigma^\lambda:=\lambda\Gamma_\alpha=\{z=f_\lambda(r), r\geq \lambda R_2\}$, where $f_\lambda(r):= \lambda  f_\alpha(\lambda^{-1} r)$. By Theorem \ref{thm curvature decay}, as $\lambda\to0$, $f_\lambda$ are uniformly bounded in $C^{1,1}_{loc}(0,+\infty)$. Hence after passing to a subsequence of $\lambda\to0$, $f_\lambda\to f_0$ in $C^1_{loc}(0,+\infty)$.
Here $f_0$ satisfies the minimal surface equation in the weak sense on $\R^n\setminus\{0\}$. Then it is directly verified that $f_0\equiv 0$. Since this is independent of the choice of subsequences of $\lambda\to0$, we obtain
\begin{equation}\label{3.5}
\lim_{r\to+\infty}f_\alpha^\prime(r)=0.
\end{equation}

{\bf Step 3.} By the bound on mean curvature in Theorem \ref{thm curvature decay}, in $(R_2,+\infty)$, $f_\alpha$ satisfies
\begin{equation}\label{3.6}
\frac{f_\alpha^{\prime\prime}(r)}{\left(1+|f_\alpha^\prime(r)|^2\right)^{3/2}}+\frac{n-1}{r}\frac{f_\alpha^\prime(r)}
{\left(1+|f_\alpha^\prime(r)|^2\right)^{1/2}}
=O\left(r^{-2}\left(\log\log r\right)^2\right).
\end{equation}
Combining this equation with \eqref{3.5}, an ordinary differential equation analysis leads to the first three estimates in \eqref{decay 1}.

The estimate on $ |f_\alpha^{(3)}(r)|$ and $ |f_\alpha^{(4)}(r)|$ follows by differentiating \eqref{3.6} in $r$, see \cite[Section 7]{18} for details. In fact, arguing as there, we can get some improved estimates, i.e. faster decay for $ |f_\alpha^{(3)}(r)|$ and $ |f_\alpha^{(4)}(r)|$, but these will not be needed in this paper.
\end{proof}
Two corollaries follow from this lemma. First the bound on $f_\alpha^\prime$ implies an area growth bound.
\begin{coro}\label{coro area growth}
There exists a constant $C$ such that for each $\Gamma_\alpha$,  if $R$ is large enough,
\[\mbox{Area}\left(\Gamma_\alpha\cap\left(\mathcal{C}_R\setminus \mathcal{C}_{R_2}\right)\right)\leq CR^n.\]
\end{coro}
Next, by this lemma and Corollary \ref{coro 2.1}, we obtain
\begin{coro}\label{coro 3.1}
For each $\Gamma_\alpha$ , there exists an $\overline{R}_\alpha$   such that $u_z$ has a definite sign in the open set $\{(r,z): r>\overline{R}_\alpha, |z-f_\alpha(r)|<1\}$.
\end{coro}

The following lemma gives a growth bound of the energy localized in the domain around each end.
\begin{lem}\label{lem energy growth localized}
For any $\alpha$,  there exists a constant $C_\alpha$ such that
\[\int_{\mathcal{M}_\alpha\cap\left\{R_2<r<R \right\}}\left[\frac{1}{2}|\nabla u|^2+W(u)\right]\leq C_\alpha R^n, \quad \forall R>R_2.\]
\end{lem}
\begin{proof}
This growth bound follows from the following two estimates.

{\bf Claim 1.} For any   $L>0$ and $R>0$,
\[\int_{\left\{R_2<r<R, f_\alpha(r)-L<z<f_\alpha(r)+L\right\}}\left[\frac{1}{2}|\nabla u|^2+W(u)\right]\leq C_\alpha R^n.\]
This follows by combining the trivial  bound $\frac{1}{2}|\nabla u|^2+W(u)\leq C$, co-area formula and the
area growth bound in Corollary \ref{coro area growth}.

{\bf Claim 2.} If $L$ is sufficiently large,
\[\int_{\left\{R_2<r<R, f_\alpha(r)+L<z<f_\alpha^+(r)\right\}}\left[\frac{1}{2}|\nabla u|^2+W(u)\right]\leq C_\alpha R^n.\]
This follows from the differential inequality
\[\Delta\left[\frac{1}{2}|\nabla u|^2+W(u)\right]\geq c\left[\frac{1}{2}|\nabla u|^2+W(u)\right] \quad \mbox{in } \left\{R_2<r<R, f_\alpha(r)+L<z<f_\alpha^+(r)\right\}.\]
This is possible if we have chosen $L$ large enough so that $W^{\prime\prime}(u)\geq c$ in this domain. (Note that by Corollary \ref{coro 2.1}, away from $\cup_\alpha\Gamma_\alpha$, $u$ is close to $\pm 1$.)
\end{proof}

\section{A Toda system}\label{sec refined asymptotics}
\setcounter{equation}{0}

In this section, keeping the notations used in the previous section, $u$ denotes an axially symmetric solution of \eqref{Allen-Cahn} in $\R^{n+1}$ satisfying that, for some $R_2>0$,  it is stable outside the cylinder $\mathcal{C}_{R_2}$ and
\[\{u=0\}\setminus\mathcal{C}_{R_2}=\cup_{\alpha\in\mathbb{Z}}\Gamma_\alpha, \quad \Gamma_\alpha:=\{z=f_\alpha(r), r>R_2\},\]
where $f_\alpha\in C^4([R_2,+\infty))$ and they are increasing in $\alpha$.

\subsection{Fermi coordinates}
For each $\alpha$, the upward unit normal vector of $\Gamma_\alpha$ at $(r,f_\alpha(r))$ is
\[N_\alpha(r):=\frac{1}{\sqrt{1+|f_\alpha^\prime(r)|^2}}\left(-f_\alpha^\prime(r)\partial_r+\partial_z\right).\]
The second fundamental form of $\Gamma_\alpha$ at $(r,f_\alpha(r))$ with respect to $N_\alpha(r)$ is denoted by $A_\alpha(r)$. The principal curvatures are
\begin{equation}\label{principal curvatures}
  \begin{cases}
           &\kappa_{\alpha,i}(r)=-\frac{1}{r}\frac{f_\alpha^\prime}{\sqrt{1+|f_\alpha^\prime|^2}}, \quad  1\leq i \leq n-1, \\
     & \kappa_{\alpha,n}(r)=-\frac{f_\alpha^{\prime\prime}}{\left(1+|f_\alpha^\prime|^2\right)^{3/2}}.
 \end{cases}
\end{equation}
By Lemma \ref{lem decay rate of derivatives I}, we have
\begin{equation}\label{bound on second fundamental form}
|A_\alpha(r)|\leq Cr^{-3/2}, \quad \forall r\geq R_2.
\end{equation}

Let $(r,t)$ be the Fermi coordinates with respect to $\Gamma_\alpha$, that is, for any point $X$ lying in a neighborhood of $\Gamma_\alpha$, take $(r,f_\alpha(r))\in\Gamma_\alpha$ to be the nearest point to $X$ and $t$ be the signed distance of $X$ to $\Gamma_\alpha$. By Theorem \ref{thm curvature decay}, these are well defined in the open set $\{(r,t): |t|<c_F r, r>R_2\}$ for a constant $c_F>0$. For each $t$, let $\Gamma_\alpha^t$ be the smooth hypersurface where the signed distance to $\Gamma_\alpha$ equals $t$. The mean curvature of $\Gamma_\alpha^t$ has the form
\begin{eqnarray}\label{mean curvature expansion}
H_\alpha(r,t)=\sum_{i=1}^{n}\frac{\kappa_{\alpha,i}(r)}{1-t\kappa_{\alpha,i}(r)}
&=&H_\alpha(r)+ O\left(|t||A_\alpha(r)|^2\right)\\
&=&H_\alpha(r) + O\left(|t|r^{-3}\right), \nonumber
\end{eqnarray}
where in the last step we have used \eqref{bound on second fundamental form}.

Denote by $\Delta_{\alpha,t}$ the Beltrami-Laplace operator with respect to the induced metric on $\Gamma_\alpha^t$. In Fermi coordinates the Euclidean Laplace operator has the following form
\begin{equation}\label{Laplacian}
\Delta=\Delta_{\alpha,t}-H_\alpha(r,t)\partial_t+\partial_{tt}.
\end{equation}

Concerning the error between $\Delta_{\alpha,t}$ and $\Delta_{\alpha,0}$,  we have (see  \cite{9, 17})
\begin{lem}\label{lem error of Laplace}
Suppose $\varphi$ is a $C^2$ function of $r$ only, then
\begin{equation}\label{error of Laplace}
\big|\Delta_{\alpha,t}\varphi(r)-\Delta_{\alpha,0}\varphi(r)\big|\leq Cr^{-3/2}\left(|\varphi^{\prime\prime}(r)|+|\varphi^\prime(r)|\right).
\end{equation}
\end{lem}
Note that here, in order to get $r^{-3/2}$ in the right hand side of \eqref{error of Laplace}, we  have used Lemma \ref{lem decay rate of derivatives I} and the estimate \eqref{bound on second fundamental form} again.

We introduce some notations.
\begin{itemize}
  \item
For $r>R$, let
$D_\alpha^\pm(r)$ be the  distance of $(r,f_\alpha(r))$ to $\Gamma_{\alpha\pm 1}$, respectively.
  \item Denote
$D_\alpha(r):=\min\left\{D_\alpha^+(r),D_\alpha^-(r)\right\}$.
\item $M(r):=\max_\alpha\max_{s\geq r}e^{-\sqrt{2}D_\alpha(s)}$.

  \end{itemize}
By Lemma \ref{lem decay rate of derivatives I},  $\Gamma_\alpha$ and $\Gamma_{\alpha+1}$ are almost parallel. Proceeding as in the proof of \cite[Lemma 8.3]{17} and \cite[Lemma 9.3]{18} we get
\begin{lem}\label{lem comparison of distances}
For any $r>R_2$,
\begin{equation*}
  \begin{cases}
     & D_\alpha^+(r)=f_{\alpha+1}(r)-f_\alpha(r)+O\left(r^{-1/6}\right), \\
     & D_\alpha^-(r)=f_\alpha(r)-f_{\alpha-1}(r) +O\left(r^{-1/6}\right).
  \end{cases}
\end{equation*}
\end{lem}

\subsection{Optimal approximation}

Fix a function $\zeta\in C_0^\infty(-2,2)$ with $\zeta\equiv 1$ in $(-1,1)$, $|\zeta^\prime|+|\zeta^{\prime\prime}|\leq 16$. For all $r$ large, let (to ease notation, dependence on $r$ will not be written down)
\[\bar{g}(t)=\zeta\left(8(\log r)t\right)g(t)+\left[1-\zeta\left(8(\log r)t\right)\right]\mbox{sgn}(t), \quad t\in(-\infty,+\infty).\]
In particular, $\bar{g}\equiv 1$ in $( 16\log r,+\infty)$ and $\bar{g}\equiv -1$ in $(-\infty, - 16\log r)$.

Note that $\bar{g}$ is an approximate solution to the one dimensional Allen-Cahn equation, that is,
\begin{equation}\label{1d eqn}
\bar{g}^{\prime\prime}(t)=W^\prime\left(\bar{g}(t)\right)+\bar{\xi}(t),
\end{equation}
where $\mbox{spt}(\bar{\xi})\in\{8\log r<|t|<16\log r\}$, and $|\bar{\xi}|+|\bar{\xi}^\prime|+|\bar{\xi}^{\prime\prime}|\lesssim r^{-4}$. Here and below we  use  the notation $\lesssim $ to mean  having an upper bound of the order of the quantity.

In the following we assume $u$ has the same sign as $(-1)^\alpha$ between $\Gamma_\alpha$ and $\Gamma_{\alpha+1}$.
\begin{lem}\label{lem orthogonal condition}
For any $r>R_2$ (perhaps after enlarging $R_2$) and $\alpha\in\mathbb{Z}$, there exists a unique $h_\alpha(r)$ such that in the Fermi coordinates with respect to $\Gamma_\alpha$,
\[\int_{-\infty}^{+\infty}\left[u(r,t)-g_\ast(r,t)\right]\overline{g}^\prime\left(t-h_\alpha(r)\right)dt=0,\]
where for each $\alpha$, in $\mathcal{M}_\alpha$ we define
\[g_\ast(r,t):= g_\alpha+\sum_{\beta<\alpha}\left[g_\beta-(-1)^\beta\right]+\sum_{\beta>\alpha}\left[g_\beta+(-1)^\beta\right],\]
and in the Fermi coordinates $(r,t)$ with respect to $\Gamma_\beta$,
\[g_\beta(r,t):=\bar{g}\left((-1)^\beta\left(t-h_\beta(r)\right)\right).\]
Moreover, for any $\alpha\in\mathbb{Z}$,
\[\lim_{r\to+\infty}\left(|h_\alpha(r)|+|h_\alpha^\prime(r)|+|h_\alpha^{\prime\prime}(r)|+|h_\alpha^{(3)}(r)|\right)=0.\]
\end{lem}
The proof of this lemma is similar to the one for \cite[Proposition 4.1]{18}, although now there may be infinitely many components. Indeed, we can define a nonlinear map on $\bigoplus_\alpha C(\Gamma_\alpha)$ as
\[F(h):=\left(\int_{-\infty}^{+\infty}\left[u(y,z)-g_\ast(y,z;h)\right]g_\alpha^\prime\left(y,z;h_\alpha\right)dz\right).\]
The $\alpha$ component of its derivative depends only on finitely many $\beta$, i.e. it has finite width. Moreover, it is diagonally dominated and hence invertible. Then this lemma follows from the inverse function theorem.

Let $g_\alpha$ and $g_\ast$ be as in this lemma. Define
$\phi:=u-g_\ast$. In Fermi coordinates with respect to $\Gamma_\alpha$, the equation for $\phi$ reads as
\begin{eqnarray}\label{error equation}
&&\Delta_{\alpha,t}\phi-H_\alpha(r,t)\partial_t\phi+\partial_{tt}\phi \nonumber\\
&=&W^{\prime\prime}(g_\ast)\phi+\mathcal{N}(\phi)+\mathcal{I} +(-1)^{\alpha}g_\alpha^\prime\mathcal{R}_{\alpha,1}-g_\alpha^{\prime\prime}\mathcal{R}_{\alpha,2}\\
&+&\sum_{\beta\neq\alpha} \left[(-1)^{\beta}g_\beta^\prime\mathcal{R}_{\beta,1}-
g_\beta^{\prime\prime}\mathcal{R}_{\beta,2}\right] -\sum_{\beta}\xi_\beta,  \nonumber
\end{eqnarray}
where
\[\mathcal{N}(\phi)=W^\prime(g_\ast+\phi)-
 W^\prime(g_\ast)-W^{\prime\prime}(g_\ast)\phi=O\left(\phi^2\right),\]
\[\mathcal{I}=W^\prime(g_\ast)-\sum_{\beta}
 W^\prime(g_\beta),\]
while for each $\beta$, in the Fermi coordinates with respect to $\Gamma_\beta$,
\[\xi_\beta(r,t)=\bar{\xi}\left((-1)^{\beta}(t-h_\beta(r))\right),\]
\[
\mathcal{R}_{\beta,1}(r,t):=H_\beta(r,t)+\Delta_{\beta,t} h_\beta(r),
\]
\[\mathcal{R}_{\beta,2}(r,t):=|\nabla_{\beta,t} h_\beta(r)|^2.\]

As in \cite[Lemma 4.6]{18}, because $u=0$ on $\Gamma_\alpha$, $h_\alpha$ can be controlled by $\phi$ in the following way.
\begin{lem}\label{control on h_0}
For each $\alpha$ and $r>R_2$, we have
\begin{equation}\label{control on h 1}
 \|h_\alpha\|_{C^{2,1/2}(r,+\infty)}\lesssim\|\phi\|_{C^{2,1/2}(\mathcal{C}_r^c)}+\max_{(r,+\infty)}e^{-\sqrt{2} D_\alpha},
\end{equation}
\begin{equation}\label{control on h_2}
\max_\alpha\| h_\alpha^\prime\|_{C^{1,1/2}(r,+\infty)} \lesssim \|\phi_r\|_{C^{1,1/2}(\mathcal{C}_r^c)}
+r^{-1/6}M(r) .
\end{equation}
\end{lem}

\subsection{Toda system}

As in \cite{9, 17}, multiplying \eqref{error equation} by $g_\alpha^\prime$ and integrating in $t$ leads to
\begin{equation}\label{Toda system}
H_\alpha+\Delta_{\alpha,0}h_\alpha
=\frac{2 A^2}{\sigma_0} \left(e^{-\sqrt{2}D_{\alpha}^-}- e^{-\sqrt{2}D_{\alpha}^+}\right)+E_\alpha,
\end{equation}
where $E_\alpha$ is a higher order term.  More precisely, we have
\begin{lem}\label{lem 5.1}
For any $ r>2R_2$,
\begin{eqnarray}\label{5.1}
|E_\alpha(r)|
&\lesssim& r^{-3}+ r^{-\frac{1}{2}}M\left(r-100 \log r\right) +M\left(r-100 \log r\right)^{\frac{4}{3}}  \nonumber\\
&+&\max_\alpha\big\|H_\alpha+\Delta_{\alpha,0}h_\alpha\big\|_{C^{1/2}(r-100\log r,+\infty)}^2+ \|\phi\|_{C^{2,1/2}(r-100\log r,+\infty)}^2.
\end{eqnarray}
\end{lem}
Here it is still useful to note that by \eqref{bound on second fundamental form}, now we can take the upper bound on  the second fundamental form to be $O\left(r^{-3/2}\right)$ when using the derivation  in \cite{18}.

\subsection{Estimates on $\phi$}

As in \cite{Gui-Liu-Wei, Wang-Wei3}, we have
\begin{lem}\label{lem estimate on error function 1}
There exist two constants $C$  such that for all $r$ large,
\begin{eqnarray*}
 &&\max_\alpha\|H_\alpha+\Delta_{\alpha,0}h_\alpha\|_{C^{1/2}(r ,+\infty)}+\|\phi\|_{C^{2,1/2}(\mathcal{C}_{r }^c)} \\
&\leq&\frac{1}{2}\left[\max_\alpha\|H_\alpha+\Delta_{\alpha,0}h_\alpha\|_{C^{1/2}(r-100\log r ,+\infty)}+\|\phi\|_{C^{2,1/2}(\mathcal{C}_{r-100\log r}^c)}\right] \\
&+&CM\left(r-100 \log r\right)+ Cr^{-3}.
\end{eqnarray*}
\end{lem}

As in \cite{9}, after finitely many times of iteration using Lemma \ref{lem estimate on error function 1}, we get a constant $C$ such that  for any $r\geq R_2$,
\[
|H_\alpha(r)+\Delta_{\alpha,0}h_\alpha(r)|+\|\phi\|_{C^{2,1/2}(\mathcal{C}_{r}^c)}\leq  C \left[r^{-3}+M\left(r-100 \log r\right)\right].
\]
By  \cite[Proposition 10.1]{18}),  $M(r)\lesssim r^{-2}\left(\log\log r\right)^2$. Hence
\begin{equation}\label{first bound}
|H_\alpha(r)+\Delta_{\alpha,0}h_\alpha(r)|+\|\phi\|_{C^{2,1/2}(\mathcal{C}_{r}^c)}\leq  C  r^{-2}\left(\log\log r\right)^2.
\end{equation}

Next by \cite[Proposition 7.1]{18},  we get
\begin{equation}\label{improved estimate}
 \|\phi_r\|_{C^{1,1/2}(\mathcal{C}_{r}^c)}\leq  C r^{-2-1/7}.
\end{equation}
In view of Lemma \ref{control on h_0}, we get
\begin{equation}\label{estimate on h}
  \|h_\alpha^\prime\|_{C^{1,1/2}((r,+\infty))}\leq  C r^{-2-1/7}.
\end{equation}

Substituting this into \eqref{Toda system} and applying Lemma \ref{lem decay rate of derivatives I}, we obtain
\begin{equation}\label{Toda 1}
f_\alpha^{\prime\prime}(r)+\frac{n-1}{r}f_\alpha^\prime(r)=\frac{2A^2}{\sigma_0}
\left[e^{-\sqrt{2}\left(f_\alpha(r)-f_{\alpha-1}(r)\right)}-e^{-\sqrt{2}\left(f_{\alpha+1}(r)-f_{\alpha}(r)\right)}\right]
+O\left(r^{-2-\frac{1}{7}}\right).
\end{equation}

By \cite[Proposition 8.1]{18}, we get the following stability condition.
\begin{prop}\label{prop stability}
  For any $\eta\in C_0^\infty (R_2,+\infty) $, we have
\begin{eqnarray}\label{reduction of stability}
&&\frac{4\sqrt{2}A^2}{\sigma_0}\int_{R_2}^{+\infty}e^{-\sqrt{2}\left(f_\alpha(r)-f_{\alpha-1}(r)\right)}\eta(r)^2r^{n-1}dr\\
&\leq&\int_{R_2}^{+\infty}\left[1+Cr^{-\frac{1}{6}}\right]|\eta^\prime(r)|^2r^{n-1}dr+C\int_{R_2}^{+\infty}\eta(r)^2r^{n-2-\frac{1}{8}}dr.\nonumber
\end{eqnarray}
\end{prop}

\section{The case $3\leq n\leq 9$: Proof of Theorem \ref{main result 1.2}}\label{sec high dim case}
\setcounter{equation}{0}

In this section we keep the same setting as in the previous section, with the additional assumption that $3\leq n\leq 9$.
In order to prove Theorem \ref{main result 1.2}, we argue by contradiction and assume there are at least two ends of $u$. We show this assumption leads to a contradiction if $3\leq n\leq 9$.

Take two adjacent ends $\Gamma_{\alpha-1}$ and $\Gamma_\alpha$. Let $v_\alpha:=f_\alpha-f_{\alpha-1}$ and $V_\alpha:=e^{-\sqrt{2}v_\alpha}$. By \eqref{Toda 1} we get a constant $\mu\in(0,1/8)$ such that
\begin{equation}\label{Toda 2}
  v_\alpha^{\prime\prime}(r)+\frac{n-1}{r}v_\alpha^\prime(r)\leq \frac{4A^2}{\sigma_0}
 e^{-\sqrt{2}v_\alpha(r) }+O\left(r^{-2-\mu}\right), \quad \mbox{in } (R_2,+\infty).
\end{equation}
Consequently,
\begin{equation}\label{Toda 3}
  -V_\alpha^{\prime\prime} -\frac{n-1}{r}V_\alpha^\prime \leq \frac{4\sqrt{2}A^2}{\sigma_0}
 V_\alpha^2-V_\alpha^{-1}\big|V_\alpha^\prime\big|^2+O\left(r^{-2-\mu}\right)V_\alpha, \quad \mbox{in } (R_2,+\infty).
\end{equation}

For any $q\in[1/2,2)$ and $\eta\in C_0^\infty(R_2,+\infty)$, multiplying \eqref{Toda 3} by $V_\alpha(r)^{2q-1}\eta(r)^2 r^{n-1}$ and integrating by parts leads to
\begin{eqnarray}\label{integral estimate 1}
&&2q\int_{R_2}^{+\infty}V_\alpha(r)^{2q-2}\big|V_\alpha^\prime(r)\big|^2\eta(r)^2r^{n-1}dr \nonumber\\
&\leq &\frac{4\sqrt{2}A^2}{\sigma_0}\int_{R_2}^{+\infty}V_\alpha(r)^{2q+1}\eta(r)^2r^{n-1}dr\\
&+& C\int_{R_2}^{+\infty}V_\alpha(r)^{2q}\left[\big|\eta^\prime(r)\big|^2+\eta(r)\big|\eta^{\prime\prime}(r)\big|+\eta(r)^2r^{-2-\mu}\right]r^{n-1}dr. \nonumber
\end{eqnarray}
On the other hand, substituting $V_\alpha^q\eta$ as test function into \eqref{reduction of stability} leads to
\begin{eqnarray}\label{integral estimate 2}
&&\frac{4\sqrt{2} A^2}{\sigma_0}\int_{R_2}^{+\infty}V_\alpha(r)^{2q+1}\eta(r)^2r^{n-1}dr \nonumber\\
&\leq& q^2\left[1+CR_2^{-\frac{1}{6}}\right]\int_{R_2}^{+\infty}V_\alpha(r)^{2q-2}V_\alpha^\prime(r)^2\eta(r)^2r^{n-1}dr\\
&+& C\int_{R_2}^{+\infty}V_\alpha(r)^{2q}\left[\big|\eta^\prime(r)\big|^2+\eta(r)\big|\eta^{\prime\prime}(r)\big|+\eta(r)^2r^{-2-\mu }\right]r^{n-1}dr. \nonumber
\end{eqnarray}

Combining \eqref{integral estimate 1} and \eqref{integral estimate 2}, if $R_2$ is sufficiently large, we get a constant $C(q)<+\infty$ such that
\begin{eqnarray}\label{integral estimate 3}
&&\int_{R_2}^{+\infty}V_\alpha(r)^{2q+1}\eta(r)^2r^{n-1}dr \\
 &\leq&  C(q)\int_{R_2}^{+\infty}V_\alpha(r)^{2q}\left[\big|\eta^\prime(r)\big|^2+\eta(r)\big|\eta^{\prime\prime}(r)\big|+\eta(r)^2r^{-2-\mu}\right]r^{n-1}dr. \nonumber
\end{eqnarray}
If $0\leq \eta \leq 1$, following Farina \cite{7},   replacing $\eta$ by $\eta^m$ for some $m\gg 1$ and then applying H\"{o}lder inequality to \eqref{integral estimate 3} we get
\begin{equation}\label{integral estimate 4}
 \int_{R_2}^{+\infty}V_\alpha(r)^{2q+1}\eta(r)^{2m}r^{n-1}dr \leq C(q)\int_{R_2}^{+\infty}\left[\big|\eta^\prime(r)\big|^2+ \big|\eta^{\prime\prime}(r)\big|+ r^{-2-\mu}\right]^{2q+1}r^{n-1}dr.
\end{equation}

For any $R>2R_2$, take $\eta_R\in C_0^\infty(R_2,2R)$ such that $0\leq \eta_R\leq 1$, $\eta_R\equiv 1$ in $(2R_2,R)$, $|\eta_R^\prime|^2+|\eta_R^{\prime\prime}|\leq 16R^{-2}$ in $(R,2R)$.
Substituting $\eta_R$ into \eqref{integral estimate 4}, we get
\begin{equation}\label{integral estimate 5}
 \int_{2R_2}^{R}V_\alpha(r)^{2q+1}r^{n-1}dr \leq C +C R^{n-2(2q+1)}.
\end{equation}
Since $n\leq 9$, we can take $2q+1=n/2$. After letting $R\to+\infty$ in \eqref{integral estimate 5} we arrive at
\begin{equation}\label{integral estimate 6}
 \int_{2R_2}^{+\infty}V_\alpha(r)^{\frac{n}{2}}r^{n-1}dr \leq C.
\end{equation}
As in Dancer-Farina \cite{dancer2009classification}, this implies that
\[\lim_{r\to+\infty}r^2e^{- \sqrt{2} v_\alpha(r)}=0,\]
which then leads to a contradiction by applying \eqref{Toda 2} exactly in the same way as in  \cite{dancer2009classification} (see also \cite{Wangstable2018}), if $n\geq 3$.

In other words, there is only one end of $u$. The one dimensional symmetry of $u$ follows by applying the main results of \cite{10} and \cite{15}, because now we have the energy growth bound from Lemma \ref{lem energy growth localized}.

\section{The case $n=2$: Proof of Theorem \ref{main result 1.1}}\label{sec 3 dim case}
\setcounter{equation}{0}

In this section $u$  denotes an axially symmetric solution of \eqref{Allen-Cahn} in $\R^3$, which is stable outside $B_{R_\ast}^2(0)\times(-R_\ast,R_\ast)$. Hence there exists a positive function $\varphi\in C^2(\R^3)$ such that
\begin{equation}\label{linearized eqn}
\Delta \varphi=W^{\prime\prime}(u)\varphi
\end{equation}
outside $B_{R_\ast}^2(0)\times(-R_\ast,R_\ast)$.

By a direct differentiation we see $u_z$ satisfies the linearized equation
\eqref{linearized eqn}. We will show
\begin{lem}\label{lem nodal domain}
Any nodal domain of $u_z$ is not disjoint from $B_{R_\ast}^2(0)\times(-R_\ast,R_\ast)$.
\end{lem}
Before proving this lemma, let us first present some technical results.

Keeping notations as in Section \ref{sec geometry of ends} and Section \ref{sec refined asymptotics}, we define
for each $\alpha$,
\[\mathcal{N}_\alpha:=\left\{X: -\frac{3}{4}D_\alpha^-\left(\Pi_\alpha(X)\right)<d_\alpha(X)<\frac{3}{4}D_\alpha^+\left(\Pi_\alpha(X)\right)\right\},\]
where $\Pi_\alpha(X)$ is the nearest point to $X$ on $\Gamma_\alpha$ and $d_\alpha$ is the signed distance to $\Gamma_\alpha$. By Theorem \ref{thm curvature decay} and Lemma \ref{lem decay rate of derivatives I}, $\Pi_\alpha$ is well defined and smooth
in the open set $\{(r,z): |d_\alpha(r,z)|<c_F r, r>R_\ast\}$ after perhaps enlarging $R_\ast$.


\begin{lem}\label{lem ends of a nodal domain}
For each $\alpha$, there exists an $R_\alpha^\ast>R_\ast$ so that the following holds.
 \begin{itemize}
 \item[(i)] There is a connected component $\Omega_\alpha$ of $\{u_z\neq0\}\cap\{r>R_\alpha^\ast\}$, which contains $\Gamma_\alpha\cap\{r>R_\alpha^\ast\}$ and is contained in $\mathcal{N}_\alpha$.
\item[(ii)] There exists a constant $C_\alpha$ such that
\begin{equation}\label{energy growth quadratic}
\int_{\Omega_\alpha\cap\mathcal{C}_R}u_z^2\leq C_\alpha R^2, \quad \forall R>R_\alpha^\ast.
\end{equation}
\end{itemize}
\end{lem}
\begin{proof}
(i) This follows by looking at the distance type function. Indeed, for any $(r_\ast,z_\ast)\in\Gamma_\alpha$ where $r_\ast$ is large, let $\varepsilon:=\max\{D_\alpha^+(r_\ast)^{-1},r_\ast^{-1}\}$ and
\[u_\varepsilon(r,z):=u\left(r_\ast+\varepsilon^{-1}r, z_\ast+\varepsilon^{-1}z\right).\]
By Proposition \ref{prop limit at infinity},
\begin{equation}\label{4.01}
\lim_{r\to+\infty}D_\alpha^\pm(r)=+\infty.
\end{equation}
Hence $\varepsilon\ll 1$ if $r_\ast\gg 1$.

Consider the distance type function $\Psi_\varepsilon$, which is defined by the relation
\[u_\varepsilon=g\left(\frac{\Psi_\varepsilon}{\varepsilon}\right).\]
By the vanishing viscosity method, as $\varepsilon\to0$, in any compact set of $\{-1\leq r \leq 1, -1\leq z \leq 1\}$, $\Psi_\varepsilon$ converges uniformly to
\makeatletter
\let\@@@alph\@alph
\def\@alph#1{\ifcase#1\or \or $'$\or $''$\fi}\makeatother
\begin{equation*}
{\Psi_\infty(r,z):=}
\begin{cases}
1-z, &1/2\leq z \leq 1, \\
z, &-1/2\leq z\leq 1/2\\
-1-z, &-1\leq z \leq -1/2.
\end{cases}
\end{equation*}
\makeatletter\let\@alph\@@@alph\makeatother
Moreover, because $\Psi_\infty$ is $C^1$ in $\{-1<r<1, -1/2<z<1/2\}$, $\Psi_\varepsilon$ converges in $C^1(\{-1<r<1, -1/2<z<1/2\})$. In particular, for all $\varepsilon$ small,
\[\frac{\partial u_\varepsilon}{\partial z}=\frac{1}{\varepsilon}g^\prime\left(\frac{\Psi_\varepsilon}{\varepsilon}\right)\frac{\partial\Psi_\varepsilon}{\partial z}<0, \quad\mbox{in }\left\{|r|<1/2, -1/4<z<1/4\right\}.\]
Similarly, $\frac{\partial u_\varepsilon}{\partial z}>0$ in $\left\{|r|<1/2, -4/5<z<-3/4\right\}\cup \left\{|r|<1/2, 3/4<z<4/5\right\}$.
Rescaling back we get the conclusion.

(ii) This follows by adding the estimates of Lemma \ref{lem energy growth localized} in $\alpha,\alpha+1$ and $\alpha-1$.
\end{proof}

\begin{lem}\label{lem Liouville}
Suppose $\Omega$ is a nodal domain of $u_z$, which is disjoint from $B_{R_\ast}^2(0)\times(-R_\ast,R_\ast)$. Then
\[\limsup_{r\to+\infty}\frac{1}{r^2}\int_{\Omega\cap B_r(0)}u_z^2=+\infty.\]
\end{lem}
\begin{proof}
Assume by the contrary, there exists a constant $C$ such that for all $r$ large,
\[ \int_{\Omega\cap B_r(0)}u_z^2 \leq Cr^2.\]
Then the standard Liouville type theorem applies to the degenerate equation (see \cite{GG,A-C})
\[\mbox{div}\left(\varphi^2\nabla\frac{u_z}{\varphi}\right)=0,\]
which implies that $u_z\equiv 0$ in $\Omega$.  This is a contradiction.
\end{proof}

\begin{proof}[Proof of Lemma \ref{lem nodal domain}]
Assume by the contrary, there is a nodal domain of $u_z$ disjoint from $B_{R_\ast}^2(0)\times(-R_\ast,R_\ast)$. Denote it by $\Omega$ and assume without loss of generality $u_z>0$ in $\Omega$. Since for any $R,r>0$,
\[|\mathcal{C}_R\cap B_r(0)|\leq CR^2r,\]
Lemma \ref{lem Liouville} implies that $\Omega$ cannot be totally contained in $\mathcal{C}_R$. In other words, $\Omega$ is unbounded in the $r$ direction.

Let $\Omega_\alpha$ be defined as in Lemma \ref{lem ends of a nodal domain}. Then we claim that\\
{\bf Claim.} There exists at most one $\alpha$ such that $\Omega_\alpha\subset\Omega$.\\
To prove this claim, we assume by the contrary that there are $\alpha\neq \beta$ such that $\Omega_\alpha\cup\Omega_\beta\subset\Omega$. Since $u_z>0$ in $\Omega_\alpha\cup\Omega_\beta$, $|\alpha-\beta|\geq 2$. In particular, there exists a $\gamma$ lying between $\alpha$ and $\beta$. Moreover, $u_z<0$ in $\Omega_\gamma$.

Let $\widetilde{\Omega}$ be the nodal domain of $u_z$ containing $\Omega_\gamma$. Viewing all of these domains as open sets in the $(r,z)$ plane, $\Omega_\alpha$ and $\Omega_\beta$ can be connected by a continuous curve totally contained in $\Omega$, which together with $\Gamma_\alpha$ and $\Gamma_\beta$ forms a simple unbounded Jordan curve. This curve divides the plane into at least two domains, $\widetilde{\Omega}$ lying on one side and $B_{R_\ast}^2(0)\times(-R_\ast,R_\ast)$ on the other side.

Then there are only finite many of ends of $u$ in $\widetilde{\Omega}$,  and we can add the estimates in Lemma \ref{lem energy growth localized} to arrive at
\[\int_{\widetilde{\Omega}\cap B_R(0)}|\nabla u|^2\leq C_{\alpha\beta}R^2, \quad \forall R \quad \mbox{large}.\]
This is a contradiction with Lemma \ref{lem Liouville}, which finishes the proof of the Claim.

By this Claim, there exists an $R_3 >0$ such that $\Omega\cap\{r>R_3 \}\subset\{f_{\alpha-1}(r)<z<f_{\alpha+1}(r)\}$. Using Lemma \ref{lem energy growth localized} again, we get a constant $C$ such that
\[\int_{ \Omega\cap B_R(0)}|\nabla u|^2\leq C R^2, \quad \forall R \quad \mbox{large}.\]
Since $\Omega$ is assumed to be disjoint from $B_{R_\ast}^2(0)\times(-R_\ast,R_\ast)$, applying Lemma \ref{lem Liouville} again we get a contradiction. This completes the proof.
\end{proof}
Since $u$ is smooth, the number of connected components of $\{u_z\neq0\}\cap B_{2R_\ast}(0)$ is finite.
Then by the above lemma we obtain
\begin{coro}
There are only finitely many nodal domains of $u_z$.
\end{coro}

Now we come to the proof of Theorem \ref{main result 1.1}.
\begin{proof}[Proof of Theorem \ref{main result 1.1}]
By the previous corollary, nodal domains of $u_z$ are denoted by $\Omega^m$, $m=1,\cdots, N$ for some $N\in\mathbb{N}$.

Assume there are infinitely many ends, $\Gamma_\alpha$. These ends are divided into $N$ classes, $\mathcal{I}_m$ ($1\leq m\leq N$), that is, $\Gamma_\alpha\in\mathcal{I}_m$ if  $\Omega_\alpha\subset\Omega^m$.

There is a class, say $\mathcal{I}_1$, containing infinitely many ends. Take two indicies $\alpha,\beta\in\mathcal{I}_1$ which are adjacent in $\mathcal{I}_1$. $\Gamma_\alpha$ and $\Gamma_\beta$ are connected by a curve in $\Omega^1$, together with $\Gamma_\alpha$ and $\Gamma_\beta$ which gives a simple unbounded Jordan curve $\gamma_{\alpha\beta}$ in the plane. This curve divides the $(r,z)$ plane into at least two open domains. Since $u_z$ has the same sign in $\Omega_{\alpha}$ and $\Omega_{\beta}$, there exists a $\Gamma_\gamma$ lying between $\Gamma_{\alpha}$ and $\Gamma_{\beta}$. Assume $\Omega_\gamma\subset\Omega^{M(\alpha)}$.
This defines a map from $\mathcal{I}_1$ to $\{1,\cdots, N\}$. Moreover, if $\alpha,\beta\in\mathcal{I}_1$ and $\alpha\neq \beta$, then $M(\alpha)\neq M(\beta)$, in other words, $\Omega^{M(\alpha)}$ and $\Omega^{M(\beta)}$ lie on two sides of a simple Jordan curve totally contained in $\Omega^1$. This leads to a contradiction because
$\mathcal{I}_1$ is an infinite set.

Once we know that there are only finitely many ends, by Lemma \ref{lem energy growth localized} we obtain a constant $C$ such that
\[\int_{B_R(0)\setminus \mathcal{C}_{R_\ast}}\left[\frac{1}{2}|\nabla u|^2+W(u)\right]\leq CR^2, \quad \forall R>R_\ast.\]
On the other hand,
\[\int_{B_R(0)\cap \mathcal{C}_{R_\ast}}\left[\frac{1}{2}|\nabla u|^2+W(u)\right]\leq C|B_R(0)\cap \mathcal{C}_{R_\ast}|\leq CR_\ast ^2R, \quad \forall R>R_\ast.\]
Combining these two estimates we get \eqref{energy growth}.

Finally, since there are only finitely many ends, by Lemma \ref{lem decay rate of derivatives I}, there exist two constants $C_4,R_4>0$ such that $\{u=0\}\setminus\mathcal{C}_{R_4}\subset\{|z|<C_4r\}$. From this we see the existence of $R>0$ such that $u$ does not change sign in $\mathcal{C}_R\cap\{|z|>R\}$.
\end{proof}

\section{Bound on number of ends: Proof of Theorems \ref{main result 2} and \ref{main result 3}}\label{sec low Morse index case}
\setcounter{equation}{0}

Since the quadratic energy growth bound has been established in Theorem \ref{main result 1.1}, the method in dimension $2$ (see \cite{Wang-Wei2}) can be extended to our setting, which gives
\begin{lem}\label{lem number of nodal domains}
Suppose $u$ is an axially symmetric solution of \eqref{Allen-Cahn} with Morse index $N\ge 1$ in $\R^3$. Then for any $e\in\R^3$, there are at most $2N$ nodal domains of $u_e:=e\cdot\nabla u$.
\end{lem}
We first use this lemma to prove Theorem \ref{main result 2}.
\begin{proof}[Proof of Theorem \ref{main result 2}]
If $u$ is stable, by Lemma \ref{lem number of nodal domains}, $u_z$ does not change sign. Then we can apply the main result in \cite{A-C} to deduce the one dimensional symmetry of $u$. Furthermore, by the axial symmetry, $u(r,z)\equiv g(z-t)$ for some $t\in\R$.
\end{proof}

Concerning solutions with Morse index $1$, we first show
\begin{lem}\label{lem 3 ends}
An axially symmetric  solution of \eqref{Allen-Cahn} with Morse index $1$ has at most three ends.
\end{lem}
\begin{proof}
If the Morse index of $u$ is $1$, by Lemma \ref{lem number of nodal domains} and Theorem \ref{main result 2}, there are exactly two nodal domains of $u_z$.

Assume there are at least $4$ ends.  Take $4$ adjacent ones, $\Gamma_\alpha$, $\alpha=1,\cdots,4$. Recall the notation $\Omega_\alpha$ defined in Lemma \ref{lem ends of a nodal domain}. Assume $u_z>0$ in $\Omega_1$ and $\Omega_3$, $u_z<0$ in $\Omega_2$ and $\Omega_4$. Since $\{u_z>0\}$ is a connected set, there is a continuous curve connecting $\Gamma_1$ and $\Gamma_3$ in $\{u_z>0\}$, which gives a simple unbounded Jordan curve contained in $\{u_z>0\}$. Clearly
$\Omega_2$ and $\Omega_4$ lies on different sides of this curve, therefore $\{u_z<0\}$ cannot be a connected set. This gives at least three nodal domains of $u_z$, a contradiction.
\end{proof}

\begin{lem}\label{lem sign of u_r}
Suppose $u$ is an axially symmetric  solution of \eqref{Allen-Cahn} with Morse index $1$. Then $u_r>0$ or $u_r<0$ strictly in $\{r\neq 0\}$.
\end{lem}
\begin{proof}
First note that $\{u_r=0\}\subset\{u_{x_1}=0\}$. Hence it cannot have interior points.
Assume by the contrary that there exist zero points of $u_r$ in $\{r\neq 0\}$. Then $\{u_{x_1}=0\}\cap\{r\neq 0\}\neq \emptyset$. Because most part of $\{u_{x_1}=0\}$ are smooth surfaces,   $\{u_{x_1}>0\}\cap\{r\neq 0\}\neq \emptyset$ and $\{u_{x_1}<0\}\cap\{r\neq 0\}\neq \emptyset$. From this and the axial symmetry we deduce the existence of two open domains $\Omega^\pm$ in the $(r,z)$ plane, where $u_r>0$ in $\Omega^+$ and $u_r<0$ in $\Omega^-$. Viewing them as open domains in $\R^3$, then $\Omega^+\cap\{x_1>0\}$ and $\Omega^-\cap\{x_1<0\}$ are two connected components  of $\{u_{x_1}>0\}$,
while $\Omega^+\cap\{x_1<0\}$ and $\Omega^-\cap\{x_1>0\}$ are two connected components  of $\{u_{x_1}<0\}$. Hence there are at least four nodal domains of $u_{x_1}$, a contradiction with Lemma \ref{lem number of nodal domains}.
\end{proof}

\begin{proof}[Proof of Theorem \ref{main result 3}]
In view of Lemma \ref{lem 3 ends}, we only need to exclude the possibility of  three ends.

By Lemma \ref{lem sign of u_r}, we can assume $u_r>0$   in $\{r\neq0\}$. Hence each connected component $\Gamma_\alpha$ of $\{u=0\}$ is a graph in the $r$-direction. There are   two cases:
\begin{itemize}
\item[{\bf Type I.}] $\Gamma_\alpha$ is not disjoint from the $z$ axis, hence it has the form $\{r=f_\alpha(z)\}$ where $f_\alpha$ is a function defined on an interval $[z_\alpha^-,z_\alpha^+)$ of the $z$ axis and $f_\alpha(z_\alpha^-)=0$;
\item[{\bf Type II.}] $\Gamma_\alpha$ is disjoint from the $z$ axis, hence it has the form $\{r=f_\alpha(z)\}$ where $f_\alpha$ is a function defined on an open interval $(z_\alpha^-,z_\alpha^+)$ of the $z$ axis.
\end{itemize}
For type I, we have $\lim_{z\to z_\alpha^+}f_i(z)=+\infty$, thus $\Gamma_\alpha$ contributes one end. For Type II, we must have $\lim_{z\to z_\alpha^\pm}f_i(z)=+\infty$, thus $\Gamma_\alpha$ contributes two ends.
Since $u$ has three ends, there are either three Type I components or one Type I plus one Type II components. Therefore $u$ can change sign one time or  three times on the $z$-axis.

{\bf Case 1.}  $u$ changes sign three times on the $z$-axis.

In this case, there is an interval $(a^-,a^+)$ such that $u(0,z)<0$ in $(a^-,a^+)$ and $u(a^-)=u(a^+)=0$. Let $\{z=f^\pm(r)\}$ be the connected components of $\{u=0\}$ emanating from $(0,a^\pm)$ respectively. Because $u_r>0$, $f^+(r)$ is decreasing in $r$
and $f^-$ is increasing. Hence
\[\lim_{r\to+\infty}\left(f^+(r)-f^-(r)\right)\leq a^+-a^-.\]
This is a contradiction with Proposition \ref{prop limit at infinity}.

{\bf Case 2.} $u$ changes sign one time on the $z$-axis.

Without loss of generality, assume $u(0,0)=0$, $u(0,z)>0$ for $z>0$ and $u(0,z)<0$ for $z<0$. There exists a connected component of $\{u=0\}$ emanating from $(0,0)$, in the form $\{z=f(r)\}$. As in Case 1, $f$ is decreasing in $r$. In particular, $u>0$ in $\{z>0\}$.
The other component of $\{u=0\}$ is Type II, which is represented by the graphs $\{z=f^\pm(r)\}$ for two functions $f^+>f^-$ defined on $[R_\ast,+\infty)$ for some $R_\ast>0$. Here $f^+$ is still increasing in $r$. As in Case 1 we get
\[\lim_{r\to+\infty}\left(f(r)-f^+(r)\right)<+\infty,\]
a contradiction  with Proposition \ref{prop limit at infinity} again.
\end{proof}

\end{document}